\documentclass[oneside,11pt]{amsart}
\usepackage{mathptmx}
\usepackage{amsthm}
\usepackage[a4paper
, scale={.75,.75}
]{geometry}
\usepackage{mathrsfs}
\usepackage{mathptmx,amssymb}
\usepackage{amsmath}
\usepackage{amsaddr}
\usepackage{color}
\usepackage{graphicx}
\usepackage[export]{adjustbox}
\usepackage{multirow}
\usepackage{mathtools}

\usepackage{enumerate}

\usepackage[dvipsnames]{xcolor}
\usepackage[colorlinks,citecolor=Plum,urlcolor=Periwinkle,linkcolor=DarkOrchid,backref=page]{hyperref}

\let\origmaketitle\maketitle
\def\maketitle{
  \begingroup
  \def\uppercasenonmath##1{} 
  \let\MakeUppercase\relax 
  \origmaketitle
  \endgroup
} 

\author{Miriam Ram\'irez \&  Ger\'onimo Uribe Bravo}
\address{Instituto de Matem\'aticas\\ 
Universidad Nacional Aut\'onoma de M\'exico}
\thanks{
Research supported by UNAM-DGAPA-PAPIIT 
grant IN114720. }

\usepackage{nameref}
\usepackage{color}
\usepackage{subcaption}
\newtheorem{theorem}{Theorem}
\newtheorem{lemma}{Lemma}
\newtheorem{proposition}{Proposition}
\newtheorem{definition}{Definition}

\usepackage[colorlinks,citecolor=Plum,urlcolor=Periwinkle,linkcolor=DarkOrchid]{hyperref}

\newcommand{\R}{\ensuremath{\mathbb{R}}}

\newcommand{\pX}{\mathbb{P}}
\newcommand{\espX}{\mathbb{E}}

\newcommand{\p}[1]{\ensuremath{\mathbb{P}\! \left( #1 \right)}}
\newcommand{\E}[1]{\ensuremath{\mathbb{E}\! \left[ #1 \right]}}

\newcommand{\floor}[1]{\ensuremath{\lfloor #1 \rfloor}}
\newcommand{\N}{\ensuremath{\mathbb{N}}}

\newcommand{\ep}{\ensuremath{\varepsilon}}

\newcommand{\fil}{\ensuremath{\mathscr{F}}}
\newcommand{\Fil}{\ensuremath{\widehat{\mathscr{F}}}}
\DeclareMathOperator{\Id}{Id}
\DeclareMathOperator{\I}{I}

\newcommand{\cadlag}{c\`adl\`ag }

\usepackage[mathlines]{lineno}

\title{The Sticky L\'evy Process  as a solution to a Time Change Equation}
\subjclass[2010]{
60G51
, 60G17
, 34F05
}

\begin{document}
    \begin{abstract}
		Stochastic Differential Equations (SDEs) were originally devised by It\^o  to provide a pathwise construction of diffusion processes. 
		A less explored approach to represent them is through Time Change Equations (TCEs) as put forth by Doeblin. 
		TCEs are a generalization of Ordinary Differential Equations driven by random functions. 
		We present  a simple example where TCEs have some advantage over SDEs.

		We represent sticky L\'evy processes 
		as the unique solution to a TCE driven by a L\'evy process with no negative jumps. 
		The solution is adapted to the time-changed filtration of the L\'evy process driving the equation. 
		This is in contrast to the SDE describing sticky Brownian motion, 
		which is known to have no adapted solutions as first proved by Chitashvili. 
		A known consequence of such non-adaptability for SDEs is that certain natural approximations to the solution of the corresponding SDE do not converge in probability, even though they do converge weakly. 
		Instead, we provide strong approximation schemes for the solution of our TCE (by adapting Euler's method for ODEs), 
		whenever the driving L\'evy process is strongly approximated. 
	\end{abstract}	

\maketitle


	\section{Introduction and statement of the results}
		Feller's discovery of sticky boundary  behavior for Brownian motion on $[0,\infty)$ (in \cite{MR0047886,MR0063607}) is, undoubtedly, a remarkable achievement. 
		The discovery is inscribed in the problem of describing
		every diffusion processes on $[0,\infty)$ that behaves as a Brownian motion 
		up to the time the former first hits $0$. 
		See \cite{MR3271518} for a historical account and \cite{MR154338} for probabilistic intuitions and constructions. 
		We now consider a definition for sticky L\'evy processes associated L\'evy processes which only jump upwards 
		(also known as Spectrally Positive L\'evy process and abbreviated SPLP). 
		General information on SPLPs can be consulted in \cite[Ch. VII]{MR1406564}.

		\begin{definition}
		Let $X$ be a SPLP and $X^0$ stand for $X$ killed upon reaching zero. 
	 	An \emph{
        extension} of $X^0$ will be \cadlag\ a strong Markov process $Z$ with values in $[0,\infty)$ 
		such that $X$ and $Z$ have the same law if killed upon reaching $0$. 
	 	We say that $Z$ is a \emph{L\'evy process with sticky boundary at $0$ based on $X$} 
		(or a sticky L\'evy process for short) 
		if $Z$ is an extension of $X^0$ for which $0$ is regular and instantaneous and which spends positive time at zero. In other words, 
		if $Z_0=0$ then 
		\[
		0=\inf\{t>0:Z_t=0\}=\inf\{t>0: Z_t\neq 0\}
		\quad\text{and}\quad
		\int_0^\infty \I(Z_s=0)\,ds >0\quad\text{almost surely. }
		\]
	 	\end{definition}

	 	It is well known that sticky Brownian motion satisfies a stochastic differential equation (SDE) of the form
	 	\begin{equation}\label{E:sde}
	 	Z_t=z+\int_0^t \I(Z_s>0)\, dB_s+\gamma \int_0^t \I(Z_s=0)\, ds,\quad t\geq 0, 
	 	\end{equation}
	 	where $B$ is a standard Brownian motion, 
	 	the stickiness parameter $\gamma$ is strictly positive and $\I$ denotes the indicator function.
	 	This equation has no strong solutions, 
		which means that any process satisfying \eqref{E:sde} involves some extra randomness to that of Brownian motion $B$. 
	 	This result was conjectured by Skorohod and initially proved by R. Chitashvili in \cite{chitashviliTR} (later published as \cite{MR1639096})
		and \cite{MR1478711}. More recent proofs can be found in \cite{MR3271518, MR3183576} and \cite{MR3607798}. 
	 	In contrast to the representation of the sticky Brownian motion as a solution to an SDE, 
		we propose a representation of any SPLP with a sticky boundary as a solution to a TCE. 
	 	The particularity of our representation is that it does not require any extra randomness to that generated by the L\'evy process driving the equation. 
	 	In the L\'evy process case, a fundamental hypothesis to construct sticky L\'evy processes will be  that the sample paths have unbounded variation on any interval. Equivalently, we can assume that either there is a Gaussian component or the sum of jumps is absolutely divergent (i.e. $\sum_{s\leq t}|X_s-X_{s-}|=\infty$ almost surely for some $t>0$). 
	 	
	 	\begin{theorem}\label{T:sLpasSolutiontoaTCE}
	 		Let $X$ be a SPLP adapted to a  right-continuous and complete filtration $(\fil_t, t\geq 0)$. 
			Assume that the sample paths of $X$ have unbounded variation.  
	 		Given a parameter $\gamma > 0$ and a point $z\geq 0$, 
			there exists a unique pair of stochastic processes 
			$Z=(Z_t, t\geq 0)$ and $C=(C_t,t\geq 0)$  satisfying 
	 		\begin{equation}\label{E:cdt}
	 		Z_t=z+X_{C_t}+\gamma\int_0^t \I(Z_s=0)\, ds, \quad 
	 		\text{ where }\quad C_t=\int_0^t \I(Z_s>0)\, ds, 
	 		\end{equation}
	 		for every 	 $t\geq 0$.
	 		For the unique pair $(Z,C)$ verifying Equation \eqref{E:cdt}, 
			it holds that $C$ is a $(\fil_t)$-time change 
			and that $Z$ is adapted to the time-changed filtration $(\Fil_t, t\geq 0)$ given by $\Fil_t=\fil_{C_t}$. 
	 		Furthermore, $Z$ is a sticky L\'evy process based on $X$. 
	 	\end{theorem} 
	 	
	 	This result attempts to honor the memory of Wolfgang Doeblin, the pioneer of TCEs, because for historical reasons that can be consulted in \cite{MR1885582}, the representation of diffusion processes suggested by Doeblin using TCEs is less known than the one given by Kiyosi It\^{o} via SDEs. 
	 	 In particular, the region of applicability of TCEs has not been as carefully delineated as the one for SDEs. Note, however, that TCEs a priori do not even need the notion of a stochastic integral to be stated and, as showed in \cite{MR3689968,MR3098685}, TCEs have much better stability properties than SDEs. 
	 	
	 	To explain the unbounded variation assumption, 
	 	it implies that the Dini derivatives of $X$ are infinite (as proved originally in \cite{MR0242261}; see \cite{MR4130409} for an extension and further applications). 
	 	In other words, at any given stopping time $T$ (such as the hitting time of zero), we have\[
	 	    -\liminf_{h\to 0+}\frac{X_{T+h}-X_T}{h}
	 	   =\limsup_{h\to 0+}\frac{X_{T+h}-X_T}{h}
	 	   =\infty. 
	 	\]This will aid in proving that $0$ is regular and instantaneous for $Z$. 
		The following (counter)example also indirectly shows its relevance: 
		the equation
	 	\[
	 	    h(t)=\beta \int_0^t \I (h(s)>0)\, ds+\gamma \int_0^t \I (h(s)=0)\, ds
	 	\]does not admit solutions if $\beta<0<\gamma$. 
		The difficulty with a time-change equation such as \eqref{E:cdt} is the discontinuity of the indicator functions of $(0,\infty)$ and of $\{0\}$. The success in its analysis follows from an explicit description of a solution in terms of reflection in the sense of Skorohod. This is done for a deterministic version of \eqref{E:cdt} in Proposition \ref{P:existencia} of Section \ref{S:existence}. 

        Sticky L\'evy processes are a one parameter family of processes built from the trajectories of $X$ and are part of the notion of recurrent extensions of $X^0$ analyzed  in \cite{Ramirez2022}
		in terms of three non-negative constants and a measure on $ (0, \infty) $. 
		Such processes are called SPLP (with values) in $[0,\infty)$. 
	 	As in Feller's  result, 
		these parameters describe the domain of the infinitesimal generator $\mathcal{L}$ of the corresponding recurrent extension.
	 	A possible boundary condition describing such a domain  is given by
		\[
			f'(0+)=\gamma^{-1} \mathcal{L}f(0+)
		\] for some constant $\gamma>0$. 
	 	In the Brownian case, this condition corresponds to  the so-called sticky Brownian motion with stickiness parameter $\gamma$. 
	 	Generalizing the Brownian case, 
	 	we will compute the boundary condition for the generator of the sticky L\'evy process of Theorem \ref{T:sLpasSolutiontoaTCE} in Section \ref{S:MartingaleProblem}. 
		Generator considerations are also relevant to explain the assumption on $X$ having no negative jumps: 
		The generator $\mathcal{L}$ of such a L\'evy process acts on functions defined on $\mathbb{R}$, 
		but immediately makes sense on functions only defined on $[0,\infty)$. 
		This last assertion is not true for the generator of a L\'evy process with jumps of both signs. 
	 	

	 	Our second main result exposes a positive consequence of the adaptability of the solution to the TCE \eqref{E:cdt}.
	 	In \cite{MR3183576}, an equivalent system  to the SDE \eqref{E:sde} is studied.
	 	In particular, it is showed that the non-existence of strong solutions 
	 	prevent the convergence in probability of certain natural approximations to the solutions of the corresponding SDE, even though they converge weakly.
	 	In contrast, we present a simple (albeit strong!) approximation scheme for the solution to the TCE \eqref{E:cdt}. 
	 	To establish such a convergence result, we start from an approximation to the L\'evy process $X$ which drives the TCE \eqref{E:cdt}. 

	 	\begin{theorem}\label{T:Euler}
	 		Let $X$ be a SPLP with unbounded variation. 
	 		Let $(Z,C)$ denote the unique solution to the TCE \eqref{E:cdt}. 
	 		Consider $(X^n,n\geq 1)$ a sequence of processes with \cadlag paths, such that each $X^n$ is the piecewise constant extension of some discrete-time process defined on $\N/n$ and starts at $0$.
	 		Suppose that $X^n\to X$ in the Skorohod topology, either weakly or almost surely. 
	 		Let $(z_n,n\geq 1)$ be a sequence of non-negative real numbers converging to a point $z$.  
			Consider the processes $C^n$ and $Z^n$ defined  by $C^n(0)=C^n(0-)=0$, 
	 	\begin{linenomath}
	 		\begin{align}
			\label{E:discretizationC}
	 		C^{n}(t)&=C^n(\floor{nt}/n-)+(t-\floor{nt}/n) \I(Z^n(t)>0)
	 		\intertext{and}
	 	\label{E:discretizationZ}
	 		Z^n(t)&= (z_n+X^n-\gamma\Id) (C^n(\floor{nt}/n))+\gamma\floor{nt}/n. 
	 		\end{align}
	 	\end{linenomath}	
	 		Then $C^n\to C$  uniformly on compact sets and $Z^n \to Z$ in the  Skorohod topology. 
	 		The type of convergence will be weak or almost sure, 
            depending on the type of convergence of $(X^n,n\geq 1)$. 	
	 	\end{theorem} 				  

Observe that the above procedure corresponds to an Euler-type approximation for the solution to the TCE \eqref{E:cdt}.
If we consider the same equation but now driven by a process for which we could not guarantee the existence of a solution,  our approximation scheme might converge but the limit might not be solution, as shown in the following simple but illustrative example.
Let $X=-\Id$, $z=0$ and $\gamma=1$. Then the approximations proposed in \eqref{E:discretizationC} and \eqref{E:discretizationZ} reduce to \[C^n\left({\frac{2k-1}{n}}\right)=C^n\left({\frac{2k}{n}}\right)=\frac{k}{n} \quad \text{ and } \quad Z^n\left(\frac{k}{n}\right)=\begin{cases} 0 & \text{ if k is even }\\
-\frac{1}{n} & \text{ if k is odd }
\end{cases}  \]
for each $k\in\N$. 
These sequences converge to $C^*(t)=t/2 $ and $Z^*=0$, but clearly such processes do not satisfy TCE \eqref{E:cdt}. In general, TCEs are very robust under approximations; the failure to converge is related to the fact that the equation that we just considered actually admits no solutions, as commented in a previous paragraph. 

Weak approximation results for sticky Brownian motion or of L\'evy processes of the sticky type have been given in \cite{MR1289173} and \cite{MR598937}. In the latter reference, reflecting Brownian motion is used, while in the former, an SDE representation is used. 
In \cite{MR4064533}, the reader will find an approximation of sticky Brownian motions by discrete space Markov chains and by diffusions in deep-well potentials as well a numerical study and many references regarding applications. 
In particular, we find there the following phrase which highlights why Theorem \ref{T:Euler} is surprising: \emph{... there are currently no methods to simulate a sticky diffusion directly: there is no practical way to extend existing methods for discretizing SDEs based on choosing discrete time steps, such as Euler-Maruyama or its variants ... to sticky processes...}
It is argued that the Markov chain approximation can be extended to multiple sticky Brownian motions. 
In the setting of multiple sticky Brownian motions, one can consult \cite{MR4161129} and \cite{MR3325271}. 
We are only aware of a strong approximation of sticky Brownian motion, in terms of  time-changed embedded simple and symmetric random walks, in \cite{MR1136247}.

		The rest of this paper is structured  as follows.
		We split the proof of Theorem \ref{T:sLpasSolutiontoaTCE} into several parts.
		In Section \ref{S:pathwise_analysis} we explore a deterministic version of the TCE \eqref{E:cdt}, 
		which is applied in Section \ref{S:monotonicity} to show a monotonicity property, the essential ingredient to show uniqueness and convergence of the proposed approximation scheme (Section \ref{S:Euler}). 
		In Section \ref{S:existence}, we obtain conditions for the existence of the unique solution to the deterministic version of the TCE \eqref{E:cdt}.
		The purpose of Section \ref{S:RandomExistenceUniquenessApproximation} is to apply the deterministic analysis to prove existence and uniqueness of the solution to the TCE \eqref{E:cdt} and the approximation Theorem \ref{T:Euler}. 
		Then in Section \ref{S:StrongMarkovProperty}, we verify that the unique process satisfying the TCE \eqref{E:cdt} is is measurable with respect to the time-changed filtration and that it is a sticky L\'evy process. 
		Finally in Section \ref{S:martingale_problem}, 
		using stochastic calculus instead of Theorem 2 from \cite{Ramirez2022}, 
		we analyze the boundary behavior of the solution to the proposed TCE to  describe the infinitesimal generator of a sticky L\'evy process.

	\section{Deterministic analysis}\label{S:pathwise_analysis}
	
		Following  the ideas from  \cite{MR3098685} and \cite{MR3689968}, 
		we start by considering a deterministic version of the TCE \eqref{E:cdt}. 
		

		We will prove that every solution to the corresponding equation satisfies a monotonicity property, which will be the key in the proof of uniqueness.
		Assume that   $Z$ solves almost surely the TCE \eqref{E:cdt}. 
		Hence,  its paths satisfy an equation of the type 
		\begin{equation}\label{E:fg}
			h(t)=f(c(t)) +g(t), \quad c(t)=\int_0^t \I( h(s)>0 )\, ds.
		\end{equation}where $f:[0,\infty)\to\R$ is a \cadlag function without negative jumps 
		starting at some non-negative value and $g$ is an non-decreasing \cadlag\ function. 
		(Indeed, we can take as $f$ a typical sample path of $t\mapsto z+X_t-\gamma t$ and $g(t)=\gamma t$.) 
		Recall that, $f$ being \cadlag, we can define the jump of $f$ at $t$, denoted $\Delta f(t)$, as $f(t)-f(t-)$. 
		By a solution to \eqref{E:fg}, we might refer either to the function $h$ (from which $c$ is immediately constructed), or to the pair $(h,c)$. 

		We first verify the non-negativity of the function $h$. 
		\begin{proposition}
            \label{P:nonNeg}
		    Let $f$ and $g$ be \cadlag\ and assume that $\Delta f\geq 0$, $g$ is non-decreasing and $f(0)+g(0)\geq 0$. 
		    Then, every solution  $h$ to the TCE \eqref{E:fg} is non-negative. 
            Furthermore, if $g$ is strictly increasing, 
            the function $c$ given by $c(t)=\int_0^t \I(h(s)>0)\, ds$ is also strictly increasing. 
		\end{proposition}
		
		\begin{proof}
			Let $h$ be a solution to \eqref{E:fg} and suppose that it takes negative values. 
            Note that $h(0)=f(0)+g(0)\geq 0$ and that $h$ is \cadlag\ without negative jumps. 
			Hence,  $h$ reaches $(-\infty, 0)$  continuously. 
			The right continuity of $f$ (and then of $h$) ensures the existence of some non-degenerate interval on which $h$ is negative. 
			Fix $\ep>0$ small enough to ensure that $\tau$ defined by
			\[\tau=\inf\{t\geq 0:  h<0 \text{ on } (t,t+\ep)\}\]
			is finite.  (Note that, with this definition and the fact that $f$ decreases continuously, we have that $h(\tau)=0$. )
			Given that $h$ is negative on a right neighborhood of $\tau$, then  
			\[\int_0^\tau \I(h(s)>0)\, ds=\int_0^{\tau+\ep} \I(h(s)>0)\, ds,\]
			which leads us to a contradiction because  
			\[0 = h(\tau)= f\left(\int_0^\tau \I(h(s)>0)\, ds\right)+g(\tau) \leq f\left(\int_0^{\tau+\ep} \I(h(s)>0)\, ds\right)+ g(\tau+\ep)=h(\tau+\ep)<0. \] 
            Hence, $h$ is non-negative. 

            Assume now that $g$ is strictly increasing. 
            By definition, $c$ is non-decreasing. 
            We prove that $c$ is strictly increasing by contradiction: 
            assume that $c(t)=c(s)$ for some $s<t$. Then, $h=0$ on $(s,t)$ and, by working on a smaler interval, 
            we can assume that $h(s)=h(t)=0$. 
            However, we then get
            \[
            0= h(s)=f\circ c(s)+g(s)<f\circ c(s)+g(t)=f\circ c(t)+g(t)=h(t)=0. 
            \]The contradiction implies that $c$ is strictly increasing. 
		\end{proof}
		If $f_-(t)=f(t-)$, note that the above result and (a slight modification of) its proof also holds for solutions to the inequality
		\[
			\int_s^t \I(h(r)>0)\, dr\leq c(t)-c(s)\leq \int_s^t \I(h(r)\geq 0)\, dr
		\] where $h(r)=f_-\circ c(r)+g_-(r)$ and $f$ and $g$ satisfy the hypotheses of Proposition \ref{P:nonNeg}. 
		These inequalities are natural when studying the stability of solutions to \eqref{E:fg} and will come up in the proof of Theorem \ref{T:Euler}. 
%
%

	\subsection{Monotonicity and Uniqueness}\label{S:monotonicity}
		The following comparison result for the solutions to Equation \eqref{E:fg} 
		will be the key idea in the uniqueness proof of Theorem \eqref{T:sLpasSolutiontoaTCE}. 
		Moreover, we pick up it in Section \ref{S:Euler}, 
		where it also plays an essential role in the approximation of sticky L\'evy processes. 
		
		\begin{proposition}\label{P:monotonia}
			Let  $(f^1,g^1)$ and  $(f^2,g^2)$ be pairs of functions satisfying that $f^i$ and $g^i$ are \cadlag, $\Delta f^i\geq 0$, $g^i$ is strictly increasing and $f^i(0)+g^i(0)\geq 0$. 
			Suppose that $f^1\leq f^2$ and $g^1\leq g^2$.
			If $h^ 1$ and  $h^2$ satisfy   
			\[  h^i(t)=f^i(c^i(t))+g^i(t), \qquad c^i(t)=\int_0^t \I(h^i(s)>0)\, ds,\]
			for $i=1,2$, then we have the inequality $c^1\leq c^2$. 
			In particular, Equation \eqref{E:fg} admits has at most one solution when $g$ is strictly increasing. 
		\end{proposition}  
		

		\begin{proof}
			Fix  
			$\ep>0$ and define 
			$c^{\ep}(t)=c^2(\ep+ t)$. 
			Set  \[\tau=\inf \{t>0:c^1(t)>c^{\ep}(t)\}.\]
			To  get a contradiction, suppose that $\tau<\infty$. 
			The continuity of $c^1$ and $c^{\ep}$ guarantees that $c^1(\tau)=c^{\ep}(\tau)$ and   $c^1$ is bigger than $c^{\ep}$ at some point  $t$ of every right neighborhood of  $\tau$. 
			At such points, the inequality $c^{\ep}(t)-c^{\ep}(\tau)<c^1(t)-c^1(\tau)$ is satisfied. 
			Applying a change of variable, this is equivalent to 
			\begin{equation}\label{E:inequality}
				\int_\tau^t \I(h^2(\ep+s)>0)\, ds <\int_\tau^t \I(h^1(s)>0)\, ds.
			\end{equation}
			The assumpions  about $g^1$ and $g^2$ imply that $g^1(\tau)<g^2(\ep+\tau)$.	
			Therefore
			\[0\leq h^1(\tau)=f^1 (c^1(\tau))+g^1(\tau) < f^2 (c^{\ep}(\tau))+ g^2(\ep+\tau)=h^2(\ep+\tau).\]
			Thanks to the right continuity of $h^2$, we can choose $t$ close enough to $\tau$ such that 
			$h^2(\ep+ s)>0$ for every $s\in[\tau, t)$. 
			Going back to the inequality \eqref{E:inequality}, 
			we see that
			\[
			t-\tau= \int_\tau^t \I(h^2(\ep+s)>0)\, ds <\int_\tau^t \I(h^1(s)>0)\, ds
			\leq t-\tau, 
			\]which is a contradiction. 
			Therefore $\tau=\infty$ and we conclude the announced result by letting  
			$\ep\to 0$. 

            In particular, if $(h^1,c^1)$ and $(h^2,c^2)$
            are two solutions to \eqref{E:fg} (driven by the same functions $f$ and $g$), 
            then the above monotonicity result (applied twice) implies $c^1=c^2$ and therefore $h^1=f\circ c^1+g=f\circ c^2+g=h^2$. 
		\end{proof}

	\subsection{Existence}\label{S:existence}
		
		The following variant of a well-known result of Skorohod (cf. \cite[Chapter VI, Lemma 2.1]{MR1725357}) will be helpful to verify the existence of the unique solution to the TCE \eqref{E:fg}.
		\begin{lemma}\label{L:existence}
			Let $f:[0,\infty)\to \R$ be a  \cadlag function with non-negative jumps and  $f(0)\geq 0$. 
			Then there exists a unique pair of functions $(r, l)$  defined on $[0,\infty)$ which  satisfies: $r=f+l$, $r$ is non-negative,
			$l$ is a non-decreasing continuous function that increases only   on the set $\{s : r(s) =0\}$ and such that $l(0)=0$.
			Moreover, the function $l$ is given by \[l(t)=\sup_{s\leq t}(-f(s)\vee 0).\]
		\end{lemma}
		Note that the lack of negative jumps of $f$ is fundamental to obtain a continuous process $l$. 
		
		With the above Lemma, we can give a deterministic existence result for equation \eqref{E:fg}. 
		\begin{proposition}\label{P:existencia}
		Assume that $f$ is \cadlag, $\Delta f\geq 0$ and $f(0)\geq 0$. 
		Let $(r,l)$ be the pair of processes of Lemma \ref{L:existence}  applied to $f$. 
		If $\{t\geq 0: r(t)=0\}$ has Lebesgue measure zero, 
		then, for every $\gamma>0$ there exists a solution $h$ to
		\begin{equation}\label{E:fgamma}
		    h=f\left(\int_0^t \I(h(s)>0) \, ds\right)+\gamma\int_0^t \I(h(s)=0)\, ds.
		\end{equation}
		\end{proposition}
		Equivalently, in terms of Equation \eqref{E:fg}, the function $h$ satisfies
		\begin{equation}\label{E:fgammaIdentidad}
			h=f^\gamma\circ c+\gamma\Id, \quad c(t)=\int_0^t \I( h(s)>0 )\, ds.
		\end{equation}
		where $f^\gamma(t)=f(t)-\gamma t$. 
		
		\begin{proof}
		Applying Lemma \ref{L:existence} to $f$, 
		we deduce the existence of a unique pair of processes $(r, l)$ satisfying $r(t)=f(t)+l(t)$ with $r$ is a non-negative function and $l$ a continuous function with non-decreasing paths such that $l(0)=0$ and
		\begin{equation}\label{E:zerosL}
			\int_0^t \I(r(s)>0) \, l(ds)=0.
		\end{equation}
		To construct the solution to the deterministic TCE \eqref{E:fgamma}, let us consider the continuous and strictly increasing function $a$ defined by $a(t)=t+l(t)/\gamma $ for every $t\geq 0$.
		Denote  its inverse by $c$ and consider the composition $h=r\circ c$. The hypothesis on $f$ implies that $\int_0^t \I(r(s)=0)\, ds=0$ for all $t$.  
		Therefore, since $r$ is non-negative, then
		\[t=\int_0^t \I(r(s)>0)\, ds=\int_0^t \I(r(s)>0)\, (ds+\gamma^{-1} l(ds)). \]
		Substituting the deterministic time $t$ for $c(t)$ in the previous expression and using that $c$ is the inverse of $a$, we have  
		\[c(t)=\int_0^{c(t)} \I(r(s)>0)\,  a(ds)=\int_0^t \I(h(s)>0)\, ds.\]
		Finally, the definition of $a$ and its continuity imply  $l(t)=\gamma(a(t)-t)$, so that
		\[l(c(t))=\gamma(t-c(t))=\gamma\int_0^t \I(h(s)=0)\, ds.\]
		Hence, the identity   $h(t)=r(c(t))$ can be written as 
		\[h(t)= f\left(\int_0^t \I(h(s)>0)\, ds\right)+\gamma \int_0^t \I(h(s)=0)\, ds,\] 
		as we wanted.
		\end{proof}


\subsection{Approximation}	
\label{S:Euler}

		It is our purpose now to discuss a simple method to approximate the solution to the TCE \eqref{E:fgamma}.
		Among the large number of existing discretization schemes,  
		we choose a widely used method, an adaptation of that of Euler's.
		Again, the key to the proof relies deeply on our monotonicity result. 

		\begin{proposition}
		\label{P:aproximacion}
		Let $f$ be \cadlag and satisfy $\Delta f\geq 0$, and $f(0)\geq 0$. 
		Assume that Equation \eqref{E:fgamma}, or equivalently \eqref{E:fgammaIdentidad}, 
		admits a unique solution denoted by $(h,c)$. 
		Let $\tilde f^n$ be a sequence of \cadlag\ functions which converge to $f$  and let $f^n=\tilde f^n-\gamma \floor{n\cdot}/n$. 
		Let $c^n$ and $h^n$ be given by $c^n(0)=c^n(0-)=0$,
		\begin{linenomath}
		\begin{align}
		    \label{E:approxfgammaCum}
		    c^{n}(t)&=c^n(\floor{nt}/n-)+(t-\floor{nt}/n) \I(h^n(t)>0)
		    \intertext{and}
		    \label{E:approxfgamma}
		    h^n(t)&= f^n (c^n(\floor{nt}/n))+\gamma\floor{nt}/n. 
		\end{align}
		\end{linenomath}
		Then $h^n\to h$ in the Skorohod $J_1$ topology and $c^n\to c$ uniformly on compact sets. 
		\end{proposition}
		
		Note that Propositions \ref{P:monotonia} and  \ref{P:existencia} give us conditions for the existence of a unique solution, 
		which is the main assumption in the above proposition. 
		Also, $h^n$ is piecewise on $[(k-1)/n,k/n)$ and, therefore, 
		$c^n$ is piecewise linear on $[(k-1)/n,k/n]$ and, 
        at the endpoints of this interval,  
		$c^n$ takes values in $\mathbb{N}/ n$. 
  Hence, $c^n(\floor{tn}/n)=\floor{nc^n(t)/n}$. 
		
		The proof of Proposition \ref{P:aproximacion} is structured as follows: 
		we prove that the sequence $(c^n, n\geq 1)$ is relatively compact.
		Given  $(c^{n_j}, j\geq 1)$ a subsequence that converges to certain limit $c^*$, we see that $((c^{n_j},h^{n_j}), j\geq 1)$ also converges and its limit is given by $(c^*,h^*)$, where $h^*=f^\gamma\circ c^*+\gamma\Id$ and we recall that $f^\gamma=f-\gamma\Id$. 
		A slight modification of the proof of Proposition \ref{P:monotonia} implies that the limit $(c^*, h^*)$ does not depend on the choice of the subsequence $(n_j,j\geq 1)$ and consequently the whole sequence $((c^n,h^n),n\geq1)$ converges.
		
		\begin{proof}[Proof of Proposition \ref{P:aproximacion}]
		Since $\gamma\Id$ is continuous, then our hypothesiss $\tilde f^n\to f$ implies that $f^n\to f-\gamma\Id$. (Since addition is not a continuous operation on Skorohod space as in \cite[Ex. 12.2]{MR1700749}, we need to use Theorem 4.1 in \cite{MR561155} or Theorem 12.7.3 in \cite{MR1876437}.)
		
		Fix $t_0>0$. 
		Note that Equation \eqref{E:approxfgammaCum} can be written as
		\[
		    c^n(t)=\int_0^t \I(h^n(s)>0)\, ds. 
		\]
		This guarantees that the functions $c^n$ are Lipschitz continuous with Lipschitz constant equal to $1$.
		Hence they are non-decreasing, equicontinuous and uniformly bounded on $[0,t_0]$.
		It follows  from  Arzel\`a-Ascoli Theorem that $(c^n, n\geq 1)$ is relatively compact. 
		Let $(c^{n_j}, j\geq 1)$ be a subsequence which converges uniformly in the space of continuous function on $[0,t_0]$, let us call  $c^*$ to the limit, which is non-decreasing and continuous. 
		Actually, $c^*$ is $1$-Lipschitz continuous, so that $c^*(t)-c^*(s)\leq t-s$ for $s\leq t$. 
		This is a fundamental fact which will be relevant to proving that $c=c^*$. 
		Since $c^{n_j}(\floor{{n_j}t}/{n_j})=\floor{{n_j}c^{n_j}(t)}/{n_j}$ for every $t\geq 0$,   we can write $h^{n_j}=f^{n_j}\circ c^{n_j}+ \gamma \floor{{n_j} \cdot}/{n_j}$.  
        We now prove that: 
        as $j\to\infty$: 
        $( c^{n_j}, f^{n_j}\circ c^{n_j})\to ( c^*, f^\gamma\circ c^*)$. 
        Indeed, the convergence $f^n\to f^\gamma$ implies that $\liminf_{n\to\infty}f^n(t_n)\geq f^\gamma_-(t)$ whenever $t_n\to t$. 
        (If a proof is needed, 
        note that Proposition 3.6.5 in \cite{MR838085} tells us that the accumulation points of $f^n(t_n)$ belong to $\{f^\gamma_-(t), f^\gamma(t)\}$.) 
        Then,
        \[
            I(f^\gamma_-\circ c^*(s)+\gamma s>0)
            \leq \liminf_j\,  \I(f^{n_j}\circ c^{n_j}(s)+\gamma\floor{ns}/n >0),
        \]so that, by Fatou's lemma, 
        \[
            \int_s^t \I(f^\gamma_-\circ c^*(r)+\gamma r>0)\, dr
            \leq c^*(t)-c^*(s). 
        \]But now, arguing as in Proposition \ref{P:nonNeg}, we see that $f^\gamma_-\circ c^*+\gamma\Id$ is non-negative and that $c^*$ is strictly increasing. Since $c^*$ is continuous and stricly increasing, Theorem 13.2.2 in \cite[p. 430]{MR561155} implies that the composition operation is continuous at $(f^\gamma,c^*)$, so that $f^{n_j}\circ c^{n_j}\to f^\gamma\circ c^*$. Since $\gamma\Id$ is continuous, we see that $h^{n_j}\to h^*:=f^\gamma\circ c^*+\gamma\Id$, as asserted. 
        
        Another application of Fatou's lemma gives
        \[
            \int_s^t \I(f^\gamma\circ c^*(r)+\gamma r>0)\, dr\leq c^*(t)-c^*(s). 
        \]Now, arguing as in the monotonicity result of Proposition \ref{P:monotonia}, we get $c\leq c^*$. 

        Let us obtain the converse inequality $c^*\leq c$ by a small adaptation of the proof of the aforementioned proposition, which then finishes the proof of Theorem \ref{T:Euler}. 
        Let $\ep>0$,  define $\tilde c(t)=c(\ep+t)$ and let $\tau=\inf\{t\geq 0: c^*(t)>\tilde c(t)\}$. 
        If $\tau<\infty$, note that $c^*(\tau)=\tilde c(\tau)$ and, 
        in every right neighborhood of $\tau$, 
        there exists $t$ such that $c^*(t)>\tilde c(t)$. 
        At $\tau$, observe that\[
        		0\leq h^*(\tau)=f^\gamma\circ c^*(\tau)+\gamma\tau<f^\gamma\circ \tilde c(\tau)+\gamma (\tau+\ep)=h(\tau+\ep). 
        \]Thanks to the right continuity of the right hand side, there exists a right neighborhood of $\tau$ on which $h(\cdot +\ep)$ is strictly positive and on which, by definition of $c$,  $\tilde c$ grows linearly. Let $t$ belong to that right-neighborhood and satisfy $c^*(t)>\tilde c(t)$. Since $c^*$ is $1$-Lipschitz continuous, we then obtain the contradiction:
        \[
        		(t-\tau)=\int_\tau^t \I(h(\ep+r)>0)\, dr=\tilde c(t)-\tilde c(\tau)<c^*(t)-c^*(\tau)\leq t-\tau. 
        \]Hence, $\tau=\infty$ and therefore $c^*\leq \tilde c$. Since this inequality holds for any $\ep>0$, we deduce that $c^*\leq c$. 

		The above implies that  $c^*=c$ and consequently  $h^*=h$. 
		In other words, the limits $c^*$ and $h^*$ do not depend on the subsequence $(n_j, j\geq 1)$ and then we conclude the convergence of the whole sequence  $((c^n,h^n), n\geq 1)$ to the unique solution to the TCE \eqref{E:fgammaIdentidad}.  
		\end{proof}

\section{Application to sticky L\'evy processes}

The aim of this section is 
to apply the deterministic analysis of the preceeding section 
to prove Theorems \ref{T:sLpasSolutiontoaTCE} and \ref{T:Euler}. 
The easy part is to obtain existence, uniqueness and approximation, 
while the Markov property 
and the fact that the solution $Z$ to Equation \eqref{E:cdt} is a sticky L\'evy process 
require some extra (probabilistic) work. 
We tackle the existence and uniqueness assertions in Theorem \ref{T:sLpasSolutiontoaTCE} and prove Theorem \ref{T:Euler} in Subsection \ref{S:RandomExistenceUniquenessApproximation}. 
Then, we prove the strong Markov property of solutions to Equation \ref{E:cdt} in Subsection \ref{S:StrongMarkovProperty}. 
This allows us to prove that solutions are sticky L\'evy processes, 
thus finishing the proof of Theorem \ref{T:sLpasSolutiontoaTCE}, 
but leaves open the precise computation of the stickiness parameter 
(or, equivalently, the boundary condition for its infinitesimal generator). 
We finally obtain the boundary condition in Subsection \ref{S:MartingaleProblem}. 
We could use the excursion analysis of \cite{Ramirez2022} 
to obtain the boundary condition 
but decided to also include a different proof via stochastic analysis to make the two works independent. 

\subsection{Existence, Uniqueness and Approximation}
\label{S:RandomExistenceUniquenessApproximation}

	    We now turn to the proof of the existence and uniqueness assertions in Theorem \ref{T:sLpasSolutiontoaTCE}. 
	    \begin{proof}[Proof of Theorem \ref{T:sLpasSolutiontoaTCE}, Existence and Uniqueness]
	    Note that uniqueness of Equation \eqref{E:cdt} is immediate from Proposition \ref{P:monotonia} by replacing the \cadlag function $f$ by the paths of $x+X-\gamma\Id$ and  taking $g=\gamma \Id$.	 
	    
	    To get existence, note that
		applying Lemma \ref{L:existence} to the paths of $X$, we deduce the existence of a unique pair of processes $(R, L)$ satisfying $R_t=z+X_t+L_t$ with $R$ a non-negative process and $L$ a continuous process with non-decreasing paths such that $L_0=0$ and
		$\int_0^t \I(R_s>0)\, dL_s=0.$
		In fact, we have an explicit representation of $L$ as 
		\begin{equation}
        \label{E:expressionforL}
			L_t=\sup_{s\leq t}((-z-X_s)\vee 0)=-\inf_{s\leq t} ((z+X_s)\wedge 0).   
		\end{equation}
		Note that $R$ corresponds to the  process $X$ reflected at its infimum which has been widely studied as a part of the fluctuation theory of L\'evy processes (cf. \cite[Ch. VI, VII]{MR1406564}, \cite{MR0386027} and \cite{MR3155252}).
		
		From the explicit description of the process $L$ given in \eqref{E:expressionforL}, 
		it follows that $\pX(R_t=0)= \pX(X_t=\underline X_t)$, where $\underline{X}_t=\inf_{s\leq t} (X_s\wedge 0)$. 
		Similarly, we denote $\overline{X}_t=\sup_{s\leq t} (X_s\vee 0)$.
		Proposition 3 from \cite[Ch. VI]{MR1406564} ensures that the pairs of variables  $(X_t-\underline{X}_t,-\underline{X}_t)$ and $(\overline{X}_t,\overline{X}_t-X_t)$ have the same distribution under $\pX$.
		Consequently 
		\[\pX (X_t=\underline{X}_t)=\pX ((X_t-\underline{X}_t, -\underline{X}_t)\in \{0\}\times [0,\infty))=\pX ((\overline{X}_t,\overline{X}_t-X_t)\in \{0\}\times [0,\infty))\leq 	\pX (\overline{X}_t=0).\]
		The unbounded variation of $X$ guarantees that $0$ is regular for $(-\infty, 0)$ and for $(0,\infty)$ (as mentioned, this result can be found in \cite{MR0242261} and has been extended in \cite{MR4130409}). 
        Hence, for any $t>0$, $\overline X_t>0$. 
		We decude that $\pX (\overline{X}_t=0)=1-\pX(X_s>0 \text{ for some } s\leq t)=0$.
		Thus,
		\[\espX\left[\int_0^{\infty}  \I(R_t=0 )\, dt \right]=\int_0^{\infty} \pX(X_t=\underline{X}_t)\, dt=0.\]
		Therefore, we can apply Proposition \ref{P:existencia} to deduce the existence of solutions to Equation \eqref{E:cdt}. 
	\end{proof}


		Let us now pass to the proof of \ref{T:Euler}. 
		
		\begin{proof}[Proof of Theorem \ref{T:Euler}]
				As we have stated  in Theorem \ref{T:Euler}, we allow the convergence $X^n\to X$ to be weak or almost surely. 
				Using Skorohod's representation Theorem, we may assume that it is satisfied almost surely in some suitable probability space. 	
				The desired result follows immediately from Proposition \ref{P:aproximacion} by considering the paths of  
				$ f=z+ X-\gamma \Id$  and $  f^n=z_n+ X^n-\gamma \floor{n\cdot}/n.$
		\end{proof}
	  				

		
		
		\subsection{ Measurability details and the strong Markov property}
		\label{S:StrongMarkovProperty}
		In order to complete the proof of Theorem \ref{T:sLpasSolutiontoaTCE}, 
  it remains to verify the adaptability of the unique solution to the TCE \eqref{E:cdt} to the time changed filtration $(\Fil_t,t\geq 0)$ and that such a solution is,  in fact, a sticky L\'evy process based on $X$. This is the objective of the current section, which ends the proof of Theorem \ref{T:sLpasSolutiontoaTCE}.

		By construction the mapping $t\mapsto C_t$ is continuous and strictly increasing. 
		Furthermore, 
        given that $C$ is the inverse of the map  $t\mapsto t+L_t/\gamma$, 
        we can write
		\[\{C_t\leq s\}=\{\gamma(t-s)\leq L_s\}\in\fil_s,\] 
		for every $t\geq 0$. In other words, the random time $C_t$ is a $(\fil_s)$-stopping time, since the filtration is right-continuous. 
		Therefore the process $C$ is a $(\fil_s)$-time change and $Z$ is adapted to the time-changed filtration $(\Fil_t,t\geq 0)$. 
		In this sense we say that $Z$ exhibits no extra randomness to that of the original L\'evy process. 
		This contrasts with the SDE describing sticky Brownian motion (cf. \cite[Theorem 1]{MR1478711}).
		
		Let us verify that the unique solution $Z$ to \eqref{E:cdt} is an extension of the killed process $X^0$. 
        By construction, we see that  if $Z_0=z>0$, then $Z$ equals $X$ until they both reach zero. Hence $Z$ and $X$ have the same law if killed upon reaching zero. 
        Let now $Z$ be the unique solution of \eqref{E:cdt} with $Z_0=z=0$. 
        The concrete construction which proves existence to \eqref{E:cdt} of Section \ref{S:existence} shows that
        \[
            \gamma \int_0^t\I(Z_s=0)\, ds= L\circ C
        \]where $C_t=\int_0^t\I (Z_s>0)\, ds$,  $L_t=-\inf_{s\leq t} X_s$. 
        We have already argued that the unbounded variation hypothesis implies that $L_t>0$ for any $t>0$ and therefore $L_\infty>0$ almost surely. 
        As above, recalling that $C$ is the inverse of $\Id +L/\gamma$, we see that $C_\infty=\infty$. We conclude that $L\circ C_\infty>0$ almost surely, so that $Z$ spends positive time at zero. 
We will now use the  unbounded variation of $X$ to guarantee the regular and instantaneous character of $0$ for $Z$.
		By construction, the unique solution $Z$ to the TCE \eqref{E:cdt} is the process $X$ reflected at its infimum by applying a continuous strictly increasing time change $C$ to it, that is $Z=R\circ C$ where $R=X-\underline{X}$.
		Consequently 
		\[\p{\inf\{s>0:Z_s=0\}=0} =\p{\inf\{s>0: X\circ C_s=\underline{X}\circ C_s \}=0}= \p{\inf\{s>0: X_s=\underline{X}_s \}=0}.\]
		Since $0$ is regular for $(-\infty,0)$  thanks to the unbounded variation hypothesis (meaning that $X$ visits $(-\infty,0)$ immediatly upon reaching $0$), we conclude the regularity of $0$.
		Similarly, given the regularity of $0$ for $(0,\infty)$ for $X$, we have  
		\[\p{\inf\{s>0:Z_s>0\}=0}= \p{\inf\{s>0: X_s>\underline{X}_s \}=0}\geq \p{\inf\{s>0: X_s>0\}=0}=1.\]
		Thus, $0$ is an instantaneous point. 
        
        To conclude the proof of Theorem \ref{T:sLpasSolutiontoaTCE}, 
        it now remains to prove the strong Markov property. 
		From the construction of the unique solution to the TCE \eqref{E:cdt}, we deduce the existence of a measurable mapping $F_s$ that maps the paths of the L\'evy process $X$ and the initial condition $z$ to the unique solution to the TCE \eqref{E:cdt} evaluated at time $s$, that is, $Z_s=F_s(X,z)$ for $s\geq  0$. 
		Let $T$ be a $(\Fil_t)$-stopping time.
		Approximating $T$ by a decreasing sequence of $(\Fil_t)$-stopping times $(T^n, n\geq 1)$ taking only finitely many values, 
		we see that $C_T$ is an $(\fil_t)$-stopping time. 
		From the TCE \eqref{E:cdt}, we deduce that 
\[			Z_{T+s}=Z_{T}+(X_{C(T+s)}-X_{C(T)})+\gamma  \int_{0}^{s} \I(Z_{T+r} =0)\, dr.\]
		Consider the processes  $\tilde{C}$, $\tilde{X}$ and $ \tilde{Z}$  given by   $\tilde{C}_s=C(T+s)-C(T)$,  $\tilde{X}_s=X_{C(T)+s}-X_{C(T)}$ and $\tilde{Z}_s=Z_{T+s}$ respectively. 
		We can write the last equation as   
		\begin{equation}\label{E:cdt_trasladada}
			\tilde{Z}_{s}=Z_T+\tilde{X}_{\tilde{C}(s)}+\gamma  \int_{0}^{s} \I(\tilde{Z}_r =0)\, dr,
		\end{equation} 
		and  $\tilde{C}$ satisfies  $\tilde{C}_s=\int_{0}^{s} \I(\tilde{Z}_r >0)\, dr$ for $s\geq 0$.
		In other words, $\tilde{Z}$ is solution to the TCE \eqref{E:cdt} driven by $\tilde{X}$ with initial condition $Z_T$.
		Consequently  $\tilde{Z}_s=F_s(\tilde{X},Z_{T})$.
		Note that $\tilde{X}$ has the same distribution as $X$  and it is independent of $\Fil_T$. 
		Hence, the conditional law of $\tilde Z$ given $\Fil_T$ is that of $F(\cdot, Z_T)$. 
		(One could make appeal to Lemma 8.7 in \cite[p. 169]{MR4226142} if needed.)
		This allows us to conclude that $Z$ is a strong Markov process and concludes the proof of Theorem \ref{T:sLpasSolutiontoaTCE}. 

	\subsection{Stickiness and martingales}\label{S:martingale_problem}
	\label{S:MartingaleProblem}
		In this section we aim at describing the boundary condition of the infinitesimal generator of the sticky L\'evy process $Z$ of Theorem \ref{T:sLpasSolutiontoaTCE} by proving the following result. 
		\begin{proposition}
		\label{P:generatorOfSLP}
		Let $X$ be a L\'evy process of unbounded variation and no negative jumps and let $\mathcal{L}$ be its infinitesimal generator.
		For a given $z\geq 0$, 
		let $Z$ be the unique (strong Markov) process satisfying the time-change equation \eqref{E:cdt}: 
		\[
			Z_t=z+X_{\int_0^t \I(Z_s>0)\, ds}+\gamma\int_0^t \I(Z_s=0)\, ds. 
		\]Then, for every $f:[0,\infty)\to \mathbb{R}$ which is of class $\mathcal{C}_{2,b}$ and which satisfies the boundary condition $\gamma f'(0+)=\mathcal{L}f(0+)$, the process $M$ defined by
		\[
			M_t=f(Z_t)-\int_0^t \mathcal{L}f(Z_s)\, ds
		\]is a  martingale and 
		\[
		\left.\frac{\partial }{\partial t}\right|_{t=0}\mathbb{E}(f(Z_t))=\mathcal{L}f(z). 
		\]
		\end{proposition}

		Theorem 2 from \cite{Ramirez2022} describes the domain of the infinitesimal generator of any recurrent extension of $X^0$ (which is proved to be a Feller process)  by means of three non-negative constants $ p_c, p_d, p_\kappa $ and a measure $ \mu $ on $ (0, \infty) $. 
		To describe such parameters we note a couple of important facts about the unique solution to \eqref{E:cdt}. 
		By construction we can see that it leaves $0$ continuously. 
		Indeed, if we consider the left endpoint $g$  of some excursion interval of $Z$, then $C_g$ is the left endpoint of some  excursion interval of the process reflected at its infimum $R$. Thanks to Proposition 2 from \cite{Ramirez2022}, such excursions start at $0$, so $Z$ leaves $0$ continuously. Thus, from \cite{Ramirez2022},  $p_c>0$ and $\mu=0$.
		Note also that $Z$ has infinite lifetime because $R$ has it and $C$ is bounded by the identity function, so $p_\kappa=0$.
		Finally, since 	$Z$  spends positive time at $0$, then $p_d>0$.	
		Theorem 2 from \cite{Ramirez2022} ensures that every 
        function $f$ in the domain of the infinitesimal 
        generator of $Z$ satisfies \[f'(0+)=\frac{p_d}{p_c} \mathcal{L}f(0+) .\]
		Our proof of Proposition \ref{P:generatorOfSLP} does not require the results from \cite{Ramirez2022}.
		The main intention is to give an application of stochastic calculus, 
		since we recall that a classical computation of the infinitesimal generator for L\'evy processes is based on Fourier analysis (cf. \cite{MR1406564}). 
        Regarding the generator $\mathcal{L}$, recall that it can be applied to $C_{2,b}$ functions such as $f$ and that $\mathcal{L}f$ is continuous (an explicit expression is forthcoming). 
        The lack of negative jumps implies that $\mathcal{L}f$ is defined even if $f$ is only defined and $C_{2,b}$ on an open set containing $[0,\infty)$. 

        \begin{proof}[Proof of Proposition \ref{P:generatorOfSLP}]
		Let $Z$ be the unique solution to the TCE \eqref{E:cdt} driven by the SPLP  $X$. 
		It\^{o}'s formula for semimartingales
		\cite[Chapter II, Theorem 32]{MR2020294} guarantees that for every function  $f\in C^2_0[0,\infty)$: 
		\begin{linenomath}
		\begin{align}
			\nonumber f(Z_t) =&f(z)+\int_0^t f'(Z^-_{s})\, dX_{C_s}+\int_0^t \gamma f'(Z^-_{s})\I(Z^-_{s}=0)\, ds+\frac{1}{2}\int_0^t f''(Z^-_{s})\, d[Z,Z]_s^c\\
			\label{E:ItoFormula}
			&+\sum_{s\leq t} ( \Delta f(Z_s)- f'(Z^-_{s}) \Delta Z_{s}). 
		\end{align} 
		\end{linenomath}
		
		In order to analyze this expression, we recall the so-called L\'evy-It\^{o} decomposition, which describes the structure of any L\'evy process in terms of three independent auxiliary L\'evy processes, each with a different type of path behaviour.
		Consider the Poisson point process $N$ of the jumps of $X$ given by
		\[N_t=\sum_{s\leq t} \delta_{(s,\Delta X_s)}.  \]
		Denote by  $\nu$ the characteristic measure of $N$, 
		which is called the L\'evy measure of $X$ and fulfills the integrability condition
		$\int_{(0,\infty)} (1\wedge x^2)\, \nu(dx)<\infty$. 
		Then, we write the L\'evy-It\^o decomposition as
		$X=X^{(1)}+X^{(2)}+X^{(3)}$,
		where $X^{(1)}=bt+\sigma B_t$ is a Brownian motion independent of $N$, 
		with diffusion coefficient $\sigma^2\geq 0$ and drift  
		$b=\espX[X_1-\int_{(0,1]} \int_{[1,\infty)}  x\, N(ds,dx)]  $,
		\[X^{(2)}=\int_ {(0,t]} \int_{[1, \infty)} x\, N(ds,dx)\]
		is a compound Poisson  process consisting of the sum of the large jumps of $X$ and finally  
		\[X^{(3)}=\int_ {(0,t]} \int_{(0,1)} x \, (N(ds,dx)-\nu(dx)ds) \]
		is a square-integrable martingale. 
		
		Assuming the L\'evy-It\^{o} decomposition of $X$ and using the next result, whose proof is postponed, we will see that  $\int_0^t	f'(Z^-_{s})\, dX_{C_s}$ is a semimartingale of the form
		\begin{equation}\label{E:integral_primera_derivada}
			M_t +\int_0^t b f'(Z^-_{s}) (1-\I(Z_s=0))\, ds+ \int_0^t f'(Z^-_{s})\, dX^{(2)}_{C_s}, 
		\end{equation}
		for some square-integrable martingale $M$. 
		
		\begin{lemma}\label{L:localMartingaleUnderTimeChange}
			Let $C$ be a $(\fil_t)$-time change whose paths are continuous and locally bounded. 
			Let $X$ be a right-continuous local martingale with respect to $(\fil_t,t\geq 0)$. 
			Then the time-changed process $X_C$ is a right-continuous local martingale with respect to the time-changed filtration $(\Fil_t,t\geq 0)$.
		\end{lemma}

		Lemma \ref{L:localMartingaleUnderTimeChange} ensures that the time-changed process $(\sigma B+X^{(3)})\circ C$ remains a local martingale.
		According to Theorem 20 from \cite[Chapter II]{MR2020294}, square-integrable local martingales are preserved under stochastic integration provided that the integrand process is adapted and has \cadlag  paths.
		Consequently the stochastic integral\footnote{We use both notations  $\int H_s \, dX_s$ and $H\cdot X$ to refer to the stochastic integral.} 
		$M=f'(Z^{-})\cdot (\sigma B_{C}+X^{(3)}_{C})$ is a $(\Fil_t)$-local martingale.
		Thanks to Corollary 27.3 from \cite[Chapter II]{MR2020294}, we know that a necessary and sufficient condition for a local martingale to be a   square-integrable martingale is that its quadratic variation is integrable. 
		Let us verify that $\E{[M,M]_t}<\infty$ for every $t\geq 0$.
		Theorem 10.17 from \cite{MR542115} implies  the quadratic variation of the time-changed process coincides with the time change of the quadratic variation
		\[\left[\sigma B_{C}+X^{(3)}_{C}, \sigma B_{C}+X^{(3)}_{C}\right]_t=\left[\sigma B+X^{(3)}, \sigma B+X^{(3)}\right]_{C_t}, \quad t\geq0.\]
		Given that the Brownian motion $B$ is independent of $X^{(3)}$, the quadratic variation is 
		$\sigma^2 C_t+\left[X^{(3)}, X^{(3)}\right]_{C_t}$, which is bounded by $\sigma^2 t+\left[X^{(3)}, X^{(3)}\right]_{t}$.
		Thus 
		\[\E{[M,M]_t} \leq \|f'\|^2_\infty \espX\left[\left[\sigma B+X^{(3)}, \sigma B+X^{(3)}\right]_{C_t}\right]\leq  \|f'\|^2_\infty\left(\sigma^2 t+t\int_{(-1,1)} x^2\, \nu(dx)\right)<\infty. \]
		This verifies the decomposition  \eqref{E:integral_primera_derivada}. 
		Later we will deal with the last term of this decomposition.
		
		Coming back to It\^{o}'s formula \eqref{E:ItoFormula}, 
		we need to calculate the term corresponding to the integral with respect to the continuous part of the quadratic variation of $Z$.
		First, we decompose the  variation as
		\[[Z,Z]_s=[X_{C},X_C]_s+2[X_{C},\gamma (\Id-C)]_s+\gamma^2 [\Id-C,\Id-C]_s,\]
		for every $s\geq 0$. 
		The first term is  $[X,X]_{C_s}$.
		Given the finite variation of $\gamma (\Id-C)$ and the continuity of $C$,  Theorem 26.6 from \cite{MR1876169} implies that almost surely the other two terms are zero. 
		Thereby $[Z,Z]_s=[X, X]_{C_s}$ for every $s\geq 0$ and  
		\[\frac{1}{2}\int_0^t f''(Z^-_{s})\, d[Z,Z]_s^c= \frac{1}{2}\int_0^t \sigma^2 f''(Z^-_{s})(1-\I(Z_s=0) )\, ds.\]
		Now we analyze the last term on the right-hand side from   \eqref{E:ItoFormula}, 
		which corresponds to the jump part. 
		Let us note that the discontinuities of $f\circ Z$ derive from the discontinuities of $Z$, which are caused by the jumps  of $X\circ C$, in other words
		\[\{s\leq t: |\Delta f(Z_s)|>0\}\!\subseteq\! \{s\leq t: \Delta Z_s>0\}=\{s\leq t: \Delta (X\circ C)_s>0\}.\]
		Making the change of variable $r=C_s$, the sum of the jumps in \eqref{E:ItoFormula} can be written as 
		\begin{equation}\label{E:suma_saltos_formula_ito}
			\sum_{r\leq C_t}(\Delta f(Z\circ A_r)-f'(Z^-\circ A_{r})\Delta (Z\circ A_r)), 
		\end{equation}
		where 	$A$ denotes the inverse of $C$.	
		We claim that $A$ is a $(\Fil_t)$-time change. 
		Indeed, splitting in the cases $r<t$ and $r\geq t$, we see that 
		$\{A_t\leq s\}\cap \{C_s\leq r\}=\{t\leq C_s\leq r\}\in\fil_r$ for any $r\geq 0$.
		Exercise 1.12 from \cite[Chapter V]{MR1725357} ensures that the time-changed filtration $(\Fil_{A_t},t\geq 0) $ is in fact  $(\fil_t,t\geq 0)$.
		Thus, for any continuous function  $g$, the process   $(g(Z^-_{ A_{t}}), t\geq 0)$ is $(\fil_t)$-predictable.	
		
		We return to \eqref{E:integral_primera_derivada} to put together the sum of the jumps in \eqref{E:suma_saltos_formula_ito} and the stochastic integral  $(f'\circ Z^-) \cdot  (X^{(2)}\circ{C})$.
		For this purpose, it is convenient to rewrite the last integral as   
		$(f'\circ Z^- \circ A\circ C)\cdot  (X^{(2)}\circ{C})$  and apply Lemma 10.18 from \cite{MR542115} to deduce that 
		$(f'\circ Z^-) \cdot  (X^{(2)}\circ{C})=((f'\circ Z^- \circ A)\cdot  X^{(2)})\circ{C}$.
		Consequently
		\begin{linenomath}
		\begin{align}\nonumber
			\int_0^{t} f'(Z^-_{s})\, dX^{(2)}_{C_s}+ &\sum_{s\leq C_t}(\Delta f(Z\circ A_s)-f'(Z^-\circ A_{s})\Delta (Z\circ A_s))\\ \nonumber
			=\int_0^{C_t} \int_{(0,\infty)}  &\left(f(Z^{-}_{A_{s}}+x)-f(Z^{-}_{A_{s}})-f'(Z^{-}_{A_{s}})x \I(x\in(0,1))\right)\, (N(ds,dx)-\nu(dx)\, ds)\\ 	\label{E:integral_saltos_primera_derivada}
			+& \int_0^{C_t} \int_{(0,\infty)} \left(f(Z^{-}_{A_{s}}+x) -f(Z^{-}_{A_{s}}) -f'(Z^{-}_{A_{s}})x\I(x\in(0,1))\right) \, \nu(dx)\, ds.
		\end{align}		
		\end{linenomath}
		Define the process $\overline{M}$ by
		\begin{linenomath}
		\begin{align*}
			\overline{M}_t =	&-\int_0^{t}  \int_{[1, \infty)}  \left(f(Z^{-}_{A_{s}}+x)-f(Z^{-}_{A_{s}})\right)\, \nu(dx)\, ds\\ 
			&+\int_0^{t}  \int_{[1, \infty)}  \left(f(Z^{-}_{A_{s}}+x)-f(Z^{-}_{A_{s}})\right)\, N(ds,dx)\\ 
			&+\int_0^{t} \int_{(0,1)}  \left(f(Z^{-}_{A_{s}}+x)-f(Z^{-}_{A_{s}})-f'(Z^{-}_{A_{s}})x\right)\, (N(ds,dx)-\, ds).
		\end{align*}
		\end{linenomath}
		Since $\nu$ is a L\'evy measure, then  
		\[\E{\int_0^{t}  \int_{[1, \infty)}  \left|f(Z^{-}_{A_{s}}+x)-f(Z^{-}_{A_{s}})\right|\, \, ds}\leq \|f\|^2_\infty t\nu([1,\infty)) <\infty.\]
		We develop the first degree Taylor polynomial of $f(Z^{-}_{A_{s}}+x)$ to obtain  
		\[f'(Z^{-}_{A_{s}})x=f(Z^{-}_{A_{s}}+x)-f(Z^{-}_{A_{s}})-R(x), \quad x\in(0,1), \]
		where the remainder $R$ satisfies $|R(x)|\leq \frac{1}{2} \|f''\|_{\infty} x^2$.
		Therefore 
		\[ \E{\int_0^{t} \int_{(0,1)}  \left(f(Z^{-}_{A_{s}}+x)-f(Z^{-}_{A_{s}})-f'(Z^{-}_{A_{s}})x\right)\, \nu(dx)\, ds} \leq \frac{1}{2}\|f''\|_{\infty} t \E{\int_{(0,1)}  x^2 \, \nu(dx)}<\infty.\]	
		Theorem 5.2.1 from  \cite{MR2512800} ensures that  $\overline{M}$ is a $(\fil_t)$-local martingale
		and Lemma \ref{L:localMartingaleUnderTimeChange} implies that $\overline{M}_C$ is a $(\Fil_t)$-local martingale. 
		Furthermore, for  $t\geq 0$ it holds that
		\[\E{\sup_{s\leq t} |\overline{M}_{C_s}|}\leq \E{\sup_{s\leq t} |\overline{M}_{s}|}\leq \left(2\|f\|_\infty+\frac{1}{2} \|f''\|^2_\infty \right)t\int_{(0,\infty)} (1\wedge x^2)\, \nu(dx)<\infty.\]
		It follows from Theorem 51 from \cite[Chapter I]{MR2020294} that   $\overline{M}_C$ is a true martingale. 
		
		Gathering all the expressions involved in It\^{o}'s formula \eqref{E:ItoFormula}, we get the semimartingale decomposition
		\begin{linenomath}
		\begin{align*}
			f(Z_t) -f(z)=&M_t+\int_0^t b f'(Z^-_{s}) (1-\I(Z_s=0))ds+\int_0^t \gamma 	f'(0+)\I(Z_{s}=0)\, ds\\
			&+\frac{1}{2}\int_0^t \sigma^2 f''(Z^-_{s})(1-	\I(Z_s=0)) \, ds+\overline{M}_{C_t}\\
			&+ \int_0^{C_t} \int_{(0,\infty)} \left(f(Z^{-}_{A_{s}}+x) -f(Z^{-}_{A_{s}}) -f'(Z^{-}_{A_{s}})x \I(x\in(0,1))\right)\,  \nu(dx) \, ds.
		\end{align*} 
		\end{linenomath}
		Recall that the extended generator of $X$ (as in \cite[Ch. VII]{MR1725357}) is given by
        \[
            \mathcal{L}f(z)=b f'(z)+ \frac{\sigma^2}{2} f''(z)
            + \int_{\R_+} \left(f(z+x) -f(z) -f'(z)x\I(x\in(0,1))\right) \, \nu(dx)
        \]on $C_{2,b}$ functions and that the extended generator of $X^0$ is given by $\mathcal{L}f$ on $C_{2,b}$ functions $f$ on $[0,\infty)$ which vanish (together with its derivatives) at $0$ and $\infty$. Note that $\mathcal{L}f(z)$ is bounded. 
        Define $\tilde{\mathcal{L}}f(0)$ by  
		\[\tilde{\mathcal{L}}f(0)= 	(b-\gamma)f'(0+)+\frac{\sigma^2}{2} f''(0+)+ \int_{\R_+} \left(f(x) -f(0+) -f'(0+)x\I(x\in(0,1))\right)\,  \nu(dx). \]
		Given that $\tilde{\mathcal{L}}f(0)=\mathcal{L}f(0+)-\gamma f'(0+)$,
		we can write the martingale $M+\overline{M}_{C}$ as
		\[M+\overline{M}_{C_t}= f(Z_t) -f(z)
        -\int_0^t \mathcal{L}f(Z^-_{s})\, ds
        +\int_0^t \tilde{\mathcal{L}}f(0)\I(Z_{s}=0)\, ds.\]
        We deduce that if a function $f\in C^2[0,\infty)$ satisfies the boundary condition $\tilde{\mathcal{L}}f(0)=0$	or equivalently	 $ \gamma f'(0+)= \mathcal{L}f(0+)$, then 
        $f(Z_t)-f(z)-\int_0^t \mathcal{L}f(Z_s)\, ds$ is a martingale. By hypothesis, the last term is bounded by a linear function of $t$, so that $\E{f(Z_t)}$ is differentiable at zero and the derivative equals $\mathcal{L}f(z)$. 
        \end{proof}
        
		We conclude this section with the proof of Lemma \ref{L:localMartingaleUnderTimeChange}.
		
		\begin{proof} \textit{(Lemma \ref{L:localMartingaleUnderTimeChange}) }
			Let $(\beta_n,n\geq 1)$ be localizing sequence for $X$, then $\beta_n \to \infty$ as $n \to \infty$ and for each  $n\geq 1$, the process $X^{\beta_n}I(\beta_n>0) $ is a uniformly integrable martingale.
			Keeping the notation $A$ for the inverse of $C$, we will prove that $(A(\beta_n),n\geq 1)$ is a sequence of $(\Fil_t)$-stopping times that localizes to  $X_C$. 
			The property of being $(\fil_t)$-stopping time is deduced by observing that 
			$\{\beta_n\leq C_t\}\in\fil_{\beta_n}\cap \fil_{C_t}\subset \Fil_t$, which implies that 
			\[\{A(\beta_n)\leq t\}\cap\{C_t\leq s\}=\{\beta_n\leq C_t\}\cap \{C_t\leq s\}\in  \fil_s. \]
			Since  $C\circ A=\Id$, then  
			\[(Z\circ C)^{A(\beta_n)}_t=Z_{C_t\wedge \beta_n}=Z^{\beta_n}_{C_t}.\] 
			Given that $Z^{\beta_n}$  is a $(\fil_t)$-martingale, Optional Stopping Theorem  guarantees that  
			\[\E{\left.Z^{\beta_n}_{C_t}\right|\fil_{C_s}}=Z^{\beta_n}_{C_s},\quad 0\leq s\leq t.\]
			Hence  $(Z\circ C)^{A(\beta_n)}$ is a $(\Fil_t)$-martingale.
			Moreover $A(\beta_n)\to \infty$ as $n\to \infty$ since $C\leq\Id$.
		\end{proof}


\end{document}